\documentclass{amsart}
\usepackage{latexsym}
\usepackage{amssymb}
\usepackage{amsmath}
\usepackage{color}
\usepackage{enumerate}
\usepackage{graphicx}
\usepackage[latin1]{inputenc}
\title{Line, spiral, dense}

\author{Neil Dobbs}\address{
Department of Theoretical Physics, University of Geneva,
24 quai Ernest-Ansermet,
1211 Gen\`eve 4, Switzerland.}
 \thanks{The author was supported by the Academy of Finland CoE in Analysis and Dynamics Research and by the ERC Bridges project, while carrying out this work.}

\newcommand\cQ{\mathcal{Q}}
\newcommand\cP{\mathcal{P}}
\newcommand\cA{\mathcal{A}}

\newcommand\cU{\mathcal{U}}

\newcommand\cbar{{\overline{\mathbb{C}}}}
\newcommand\remark{\noindent \emph{Remark: }}

\newcommand\cL{\mathcal{L}}

\newcommand\cW{\mathcal{W}}

\newcommand\re{{\Re}}
\newcommand\im{{\Im}}

\newcommand\eps{\varepsilon}

\newcommand\arr{\mathbb{R}}
\newcommand\ccc{\mathbb{C}}
\newcommand\R{\mathbb{R}}
\newcommand\Z{\mathbb{Z}}

\newtheorem{thm}{Theorem}

\newtheorem{lem}[thm]{Lemma}

\newtheorem{cor}[thm]{Corollary}

\begin{document}
\date{\today}
\begin{abstract}
    Exponential of exponential of almost every line in the complex plane is dense in the plane. On the other hand, for lines through any point, for a set of angles of Hausdorff dimension one, exponential of exponential of a line with angle from that set  is not dense in the plane.  The third iterate of an oblique line is always dense. 
\end{abstract}
    
\maketitle

\section{Introduction}
In 1914, Harald Bohr and Richard Courant showed that for the Riemann zeta function, if $\sigma \in (\frac12, 1]$, then $\overline{\zeta(\sigma + i \arr)} = \ccc$, i.e.\ the image of any vertical line with real part in $(\frac12, 1]$ is dense \cite[\S4, p.271]{bohr1914neue}. One hundred years on, we ask what happens under the exponential map. 

One may picture the exponential map, $\exp : z \mapsto e^z \in \ccc$, as mapping Cartesian coordinates onto polar coordinates, since $\exp(x + iy) = e^x e^{iy}$. It maps vertical lines to circles centred on $0$ and maps horizontal lines to rays emanating from $0$. The map is infinite-to-one and $2\pi i$-periodic;  preimages of a point lie along a vertical line.  Oblique (slanted)  lines get mapped to logarithmic spirals. 

Applying exponential a second time, what happens? See Figure~\ref{fig:zero}. Circles are compact, so their images are compact. Rays are subsets of lines, so they get mapped into circles, rays or logarithmic spirals. Intriguingly, the image of a logarithmic spiral under exponential is not obvious, and for good reason. 

\begin{figure}
    \centering
        \def\svgwidth{0.9\columnwidth}
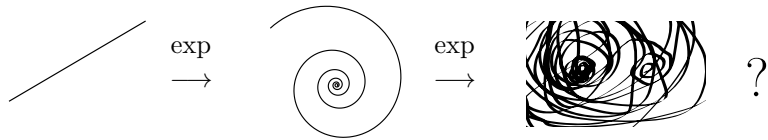
\caption{Line, spiral, what?}
    \label{fig:zero}
\end{figure}

For $p \in \ccc$, $\alpha \in \arr$, let $L_\alpha(p) := \{p + t(i + \alpha) : t  \in \arr\}$.  
Set $\cL(p) := \{L_\alpha(p) : \alpha \in \arr\}$, the family of non-horizontal lines through a point $p \in \ccc$, parametrised by $\alpha \in \R$. With this parametrisation, there is a natural one-dimensional Lebesgue measure on the set $\cL(p)$. It is equivalent to the measure obtained when parametrising the family by angle (points on the half-circle). 

\begin{thm} \label{thm:one}
        Given $p \in \ccc$, for Lebesgue almost every $\alpha \in \R$, 
        $$ 
            \overline{\exp \circ \exp(L_\alpha(p))} = \ccc.
        $$ 
        \end{thm}
        
    From the topological perspective, a property is \emph{generic} in some space if it holds for all points in a \emph{residual set}, that is, a set which can be written as a countable intersection of open, dense sets. 

\begin{thm} \label{thm:three}
            For each $p \in \ccc$, the set $\{\alpha \in \arr  : \overline{ \exp\circ \exp(L_\alpha(p))} = \ccc \}$ is residual in $\R$. 
    \end{thm}
\begin{thm} \label{thm:threesteps}
    The image of an oblique line under $\exp \circ \exp \circ \exp$ is dense in $\ccc$.
        \end{thm}
    In other words, for each $p \in \ccc$, for every $\alpha \in \R\setminus \{0\}$, 
        $$ 
            \overline{\exp \circ \exp \circ \exp(L_\alpha(p))} = \ccc.
        $$ 
        Of course, every subsequent iterate of an oblique line is also dense. 

        In general, it is hard to determine whether a given line will have dense image or not under $\exp \circ \exp$. Certain ones do, however, and we obtain a concisely defined, explicit, analytic dense curve. Let $a \in (0,1)$ be the binary Champernowne constant (with binary expansion $0.11011100101\ldots$) or any other number whose binary expansion contains all possible finite strings of zeroes and ones. Let $p_* :=  \log (2\pi a) + \frac{\pi}2 i$ and $\alpha_* := \frac{\log 2}{2\pi}$. 
\begin{thm} \label{thm:explicit}
    $\overline{\exp \circ \exp(L_{\alpha_*}(p_*))} = \ccc.$
        \end{thm}
        
    \medskip
    In Theorem~\ref{thm:one} we obtained a full-measure set of parameters with dense image. One may be tempted to think that all oblique lines would have dense image under $\exp \circ \exp$. However, this is not true, and 
    to Theorem~\ref{thm:one} there is the following
     complementary statement.

\begin{thm} \label{thm:two}
        For each $p \in \ccc$ and each open set $X \subset \arr$, the set $$\{\alpha \in X  : \overline{ \exp\circ \exp(L_\alpha(p))} \ne \ccc \}$$ has Hausdorff dimension $1$. 
    \end{thm}
    Let $Y$ denote the set of $\theta \in (1,\infty)$ for which 
    $(\theta^k)_{k\geq0}$ is not dense modulo $1$. Kahane~\cite{Kahane:mod1} proved that $Y \subset \arr$ has Hausdorff dimension $1$; however, in any bounded interval, he only obtained dimension close to $1$. In Lemma~\ref{lem:2pidense}, we establish a connection between intersections of a logarithmic spiral with the imaginary axis and density of the image of the spiral under exponential. This allows us to improve Kahane's result  a little. 
        \begin{cor}\label{cor:theta}
        For each open interval $I \subset (1, \infty)$, the set of $\theta \in I$ for which $(\theta^k)_{k\geq 0}$ is not dense modulo $1$ has Hausdorff dimension $1$. 
    \end{cor}

    \remark
    The above-described phenomena are not unique to the exponential map, the most fundamental of transcendental maps. 
    Once one understands exponential, extensions to  maps such as $z \mapsto \sin(z), \exp(z^n), \exp\circ\exp(z)$ are not hard to devise, but what of a general statement?  

    \remark 
    For a generic entire function of the complex plane, the image of the real line is dense. Indeed, 
    Birkhoff\footnote{The author thanks P.\ Gauthier for a helpful conversation in this regard.} \cite{birkhoff1929} showed the existence of an entire function $f$ whose translates $T_n f : x \mapsto f(x-n)$ 
    approximate polynomials in $\mathbb{Q}[x] + i\mathbb{Q}[x]$ arbitrarily well (on compacts). In particular, $(T_nf)_{n\in \Z}$ is dense in the (Fr\'echet) space of entire functions with the topology of uniform convergence on compacts.  Hence, given an open set $U$ of entire functions, there is some $N \in \Z$ with $T_Nf \in U$. 
     Since the translation operators $T_n$ are continuous, $\bigcup_{n\in \Z} T_n U$ is an open dense set.  Now let $\cU$ be a countable basis of open sets for the topology. The set 
    $$X_\cU := \bigcap_{U\in \cU} \bigcup_{n\in \Z} T_n(U)$$
    is residual. Consider $g \in X_\cU$. One readily checks that the translates $(T_ng)_{n\in \Z}$ enter each set in the basis and hence are dense in the space of entire functions. In particular, the translates approximate all constant functions.  Hence  $\overline{g(\arr)} = \ccc$, as required. 
    The fact that  a generic curve has dense image does not tell one what happens for a particular map or for a subfamily (for example, no logarithmic spiral is dense). Besides Birkhoff-style constructions and curves coming from things resembling $\zeta$-functions, we are unaware of other previously-known dense analytic curves. 


        \medskip
        One can also ask (in the spirit of \cite{BJ:1, BJ:2}) about the distributions of the curves considered, in the following sense. Given $\alpha, p$, let $\rho : t \mapsto \exp\circ \exp (p + t(i + \alpha))$, so $\rho$ parametrises $\exp\circ \exp$ of the line $L_\alpha(p)$. 
For every measurable set $A$ and $T>1$, let $\mu_T(A):= \frac1{2T} m(\{t \in [-T,T]: \rho(t) \in A\})$, where $m$ denotes Lebesgue measure. Then $\mu_T$ is a probability measure. With the weak$^*$-topology on the space of probability measures on $\cbar$, we obtain the following unilluminating result. 
    \begin{thm} \label{thm:weak}
        For every oblique line $L_\alpha(p)$, the corresponding measures $\mu_T$ satisfy 
        $$
        \lim_{T \to \infty} \mu_T = \frac{\delta_0}4 + \frac{\delta_1}2 +\frac{\delta_\infty}4,$$
        where $\delta_z$ denotes the Dirac mass at the point $z$.
    \end{thm}

    
    We shall use $\re(z)$ and $\im(z)$ to denote the real and imaginary parts of a complex number $z$. We denote one-dimensional Lebesgue measure by $m$ and denote the length of an interval $I$ by $m(I)$ or by $|I|$. If $\Sigma : t \mapsto \exp(p + t(i+\alpha))$, then $\frac{d}{dt}\Sigma(t) = \Sigma(t)(i+\alpha)$. Therefore the spiral $\Sigma$ has tangent of slope $-\alpha$ when it intersects the imaginary axis. 

    The proofs are provided in linear fashion. 
\section{Dense analytic curves} 
In this section we prove Theorems~\ref{thm:one}-\ref{thm:threesteps}.
    \begin{proof}[Proof of Theorem~\ref{thm:one}]
            Let $f$ denote $\exp \circ \exp$. Fix $p$ and write $L_\alpha$ for $L_\alpha(p)$.  Let 
                \begin{equation} \label{eq:XU}
                X_U := \{\alpha : f(L_\alpha) \cap U \ne \emptyset\}.
            \end{equation}
        Given a  sequence $(q_n)_{n=1}^\infty$ dense in $\ccc$ and a decreasing sequence of positive reals $(\delta_n)_{n=1}^\infty$ with $\delta_n \to 0^+$, let $\cU := \{B(q_n, \delta_n) : n\geq 1\}$. 
        Then a set is dense in $\ccc$ if and only if it has non-empty intersection with each $U \in \cU$. Since $\cU$ is countable, if for each $U \in \cU$, $X_U$ has full measure, then $X_\infty := \bigcap_{U \in \cU} X_U$ has full measure as a countable intersection of full-measure sets. Of course, for each $\alpha \in X_\infty$, $f(L_\alpha)$ is dense in $\ccc$. 

        Thus proving Theorem~\ref{thm:one}
        reduces to showing that for any open ball $U$, $X_U$ has full measure. We say a point $x$ is an $\eps$-density point for a set $X \subset \arr$ if $\lim_{r \to 0^+} \frac{m(X \cap B(x,r))}{m(B(x,r))} \geq \eps$.  
        By the Lebesgue density point theorem, almost every point of $X$ is a $1$-density point for $X$. 
        On the other hand, if $\eps >0$ and almost every point in $\arr$ is an $\eps$-density point for a set $X \subset \arr$, then the set of $1$-density points for the complement of $X$ has zero measure, so the complement has zero measure, so $X$ must have full measure. 
        It therefore suffices to prove that, 
        given a ball $U$, there exists $\eps >0$ such that each $\alpha_0 \in \arr\setminus\{0\}$ is an $\eps$-density point for $X_U$. So let us do this.

        Let $V := \exp^{-1}(U)$. Then $V$ is an open set. Let $H$ be a vertical line, with real part $h \ne 0$, which intersects $V$, see Figure~\ref{fig:one}. Since $\exp$ is $2\pi i$-periodic, $H \cap V$ contains an open interval $I$ and all $2 \pi i$-translates of $I$. In particular, for any subinterval $T \subset H$ of length at least $2\pi$, 
        \begin{equation}\label{eq:TV}
	m(T \cap V)/m(T) \geq m(I)/4\pi. 
	\end{equation}

\begin{figure}
    \centering
    \def\svgwidth{0.8\columnwidth}
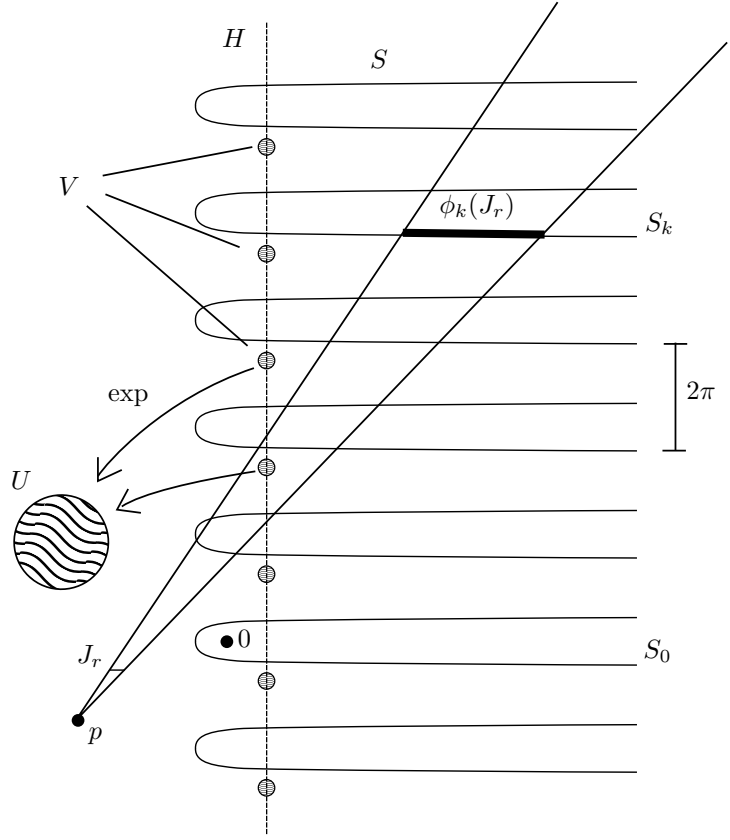
    \caption{An open ball $U$, $V = \exp^{-1}(U)$, a vertical line $H$ passing through $V$, $S = \exp^{-1}(H)$ and the projection $\phi_k$ onto a component $S_k$ of $S$.}
    \label{fig:one}
\end{figure}

        Now consider $S = \exp^{-1}(H)$. 
        If $h >0$ then one connected component of $S$, $S_0$ say, can be parametrised by 
        $$
        \gamma_+ : t \mapsto \frac12 \log(t^2 +h^2) + i \arctan \frac{t}{h}
        $$ 
        with $\gamma_+(\arr) = S_0$. 
        If $h<0$ then $S_0$ can be parametrised by $\gamma_- : t \mapsto \pi i + \gamma_+(t)$. 
        Taking the derivative of $\gamma_+$ and $\gamma_-$, 
        \begin{equation}\label{eq:slope}
        \gamma_+'(t) = \gamma_-'(t) = \frac{t}{t^2 +h^2} + i\frac{h}{t^2 + h^2}, 
    \end{equation}
        so the slope of $\gamma_\pm$ tends to $0$ as $|t| \to  \infty$. 
        For $k \in \Z$, if $\alpha_0 >0$ let
          $S_k := S_0 + 2k\pi i$; otherwise let $S_k := S_0 - 2k \pi i$.  Then $S_k$ for $k \in \Z$ are the connected components of $S$. 

        Let $W_k := S_k \cap \exp^{-1}(V)$. 
        The absolute value of the derivative of $\exp$ on $S$ is bounded below by $|h|> 0$, so any segment of $S_k$ of length at least $2\pi/|h|$ gets mapped onto a segment of $H$ of length at least $2\pi$. 
        The distortion of $\exp$ (by \emph{distortion}, we mean the ratio of the absolute value of the derivative at any two points) is bounded by $e^{4\pi/|h|}$ on each vertical strip of width $4\pi/|h|$. By the distortion bound and~\eqref{eq:TV}, 
         for any segment $B$ of $S_k$ of length between $2\pi/ |h|$ and $4\pi/|h|$, 
        \begin{equation} \label{eq:denseB}
            \frac{m(B \cap W_k)}{m(B)} \geq \frac{m(\exp(B) \cap V)}{m(\exp(B)) e^{4\pi/|h|} }\geq \frac{m(I)}{ 4 \pi e^{4\pi/ |h|}} .
        \end{equation}
	Any segment $B$ of $S_k$ of length at least $2\pi/|h|$ can be divided into segments of length between $2\pi /|h|$ and $4\pi /|h|$, so~\eqref{eq:denseB} continues to hold for all segments $B$ of $S_k$ of length at least $2\pi/|h|$. 

        Let $\xi : \alpha \mapsto p + i + \alpha$. 
        Let $\alpha_0 \in \arr \setminus\{0\}$ and let $r_0 = |\alpha_0|/2$. 
        For $r \in  (0, r_0)$, let $J_r := \xi(B(\alpha_0, r))$ be the open line segment joining the points $p + i + \alpha_0 -r$ and $p + i + \alpha_0 + r$. 
        For some $K \geq 1$, 
        for every $k \geq K$, 
        for each $\alpha \in B(\alpha_0,  r_0)$, 
        $L_\alpha$ intersects $S_k$ transversely (twice).
        For $k \geq K$, let $\phi_k$ denote the central projection with respect to $p$ from $J_{r_0}$ to $S_k$ (taking the first point of intersection). For some $K_0 > K$ and each $k \geq K_0$, $\phi_k(J_{r_0})$ is almost horizontal and the distortion of $\phi_k$ on $J_{r_0}$ is close to $1$, in particular it is bounded by $2$. 
        Now simple geometry entails that $m(\phi_k(J_r)) /\pi kr \to 1$ as $k \to \infty$ so, for each $r \in (0,r_0)$,
        there exists  $k_r  \geq K_0$ with $m(\phi_{k_r}(J_r)) > 2\pi/|h|$. Let $X_r := J_r \cap \phi_{k_r}^{-1}(W_{k_r} ).$ 
        From \eqref{eq:denseB} and the distortion bound of $2$,
         we deduce  that
         $m(X_r)/m(J_r) \geq \eps$, for $\eps := m(I)/8\pi e^{4\pi/|h|} $. 
         For $\alpha \in \xi^{-1}(X_r)$, $L_\alpha \cap W_{k_r} \ne \emptyset$ so $f(L_\alpha) \cap U \ne \emptyset.$
         In particular, $\xi^{-1}(X_r) \subset X_U$ and $$\frac{m(\xi^{-1}(X_r))}{m(B(\alpha_0, r))} \geq \eps.$$
        Noting that $\eps$ depends only on $U$ and $h$, we have shown that $\alpha_0$ is an $\eps$-density point for $X_U$ for each $\alpha_0 \in \arr \setminus \{0\}$.
    \end{proof}
    
    \begin{proof}[Proof of Theorem~\ref{thm:three}]
        Fix $p \in \ccc$. 
        Let $(q_n)_{n=1}^\infty$ be a dense sequence in $\ccc$ and let $(\delta_n)_{n=1}^\infty$ be a decreasing sequence of positive reals with $\delta_n \to 0^+$. 
        Let $\cU := \{B(q_n, \delta_n) : n \geq 1\}$. As per \eqref{eq:XU}, given an open set $U$, let $$X_U := \{\alpha : \exp \circ \exp(L_\alpha(p)) \cap U \ne \emptyset\}.$$
        Since $\exp$ is continuous (so $\exp^{-2}(U)$ is open) and the central projection is an open map, $X_U$ is open.  By Theorem~\ref{thm:one}, $X_U$ has full measure and thus is dense and open for each open set $U$. 
        Consequently, $X_\infty := \bigcap_{U \in \cU} X_U$ is a countable intersection of open, dense sets. As in the proof of Theorem~\ref{thm:one}, each point $\alpha \in X_\infty$ satisfies $\exp \circ \exp (L_\alpha(p))$ is dense. 
    \end{proof}
        
    \begin{proof}[Proof of Theorem~\ref{thm:threesteps}]
        We wish to show that the image of an oblique line under $\exp\circ\exp\circ\exp$ is dense. 
        Let us reprise the notation of the preceding proof, so $U$ is an open set, $V = \exp^{-1}(U)$, $H$ a vertical line (not containing $0$) intersecting $V$, $S = \exp^{-1}(H)$ and $S_k$ the connected components of $S$. 
        Let $v_0$ be a point in $H \cap V$ and let $v_j := v_0 + 2j\pi i$, so $v_j \in H\cap V$ for all $j \in \Z$. 
        Let $w_j^k$ denote the preimage of $v_j$ in $S_k$, and write $\omega_j$ for the real part of $w_j^k$, noting that this is independent of $k$. As $j \to \infty$, $\omega_j$ tends to $+\infty$. Therefore the slope of the line segment $Z^k_j$ joining $w_j^k$ to $w_{j+1}^k$ tends to $0$ as $j \to \infty$ (cf.~\eqref{eq:slope}). Since $H$ is a vertical line, $S_k$ lies in a horizontal strip of height $\pi$, and so 
        $$\gamma_{j_0}^k := \bigcup_{j\geq j_0} Z_j^k$$ is a curve, contained in a strip of height $\pi$, joining $w_{j_0}^k$ to $\infty$. 

        For some $r \in (0,1)$, $B(v_j, r) \subset V$. Estimating via the derivative of $\exp$, we obtain $B(w^k_j, \frac{r}{2|v_j|}) \subset \exp^{-1}(V)$ for all large $j$, and similarly that $|w^k_{j+1} - w^k_j| < 2\pi/|v_j| < 1$. Setting $\delta := r/4\pi$, we deduce that
        $$
        B_j^k := B(w^k_j, \delta |w_{j+1}^k - w_j^k|) \subset \exp^{-1}(V).$$
        
        Simple geometry then entails that if $\rho$ is a smooth curve with slope bounded in absolute value by $\delta/2$ which intersects the line segment $Z^k_j$ and whose projection onto the real line contains $(\omega_j, \omega_{j+1})$, then $\rho$ intersects $B_j^k$. This holds for all $j \geq j_0$, for some large $j_0$, independent of $k$. 

        Now any curve in the half-plane $\{\re(z) > \omega_{j_0}\}$ whose imaginary part has range  at least $3\pi$ long must intersect a curve $\gamma_{j_0}^k$ for some $k$. From this we deduce that if $\rho'$ is a smooth curve 
         contained in $\{\re(z) > \omega_{j_0}\}$, 
        with slope lying in $(\delta/4,\delta/2)$ and of horizontal length at least $2 + 12\pi/\delta$, then $\rho'$ must intersect some $B_j^k$. 
        Indeed, there is a subcurve whose projection is $(\omega_{j_1}, \omega_{j_2})$ (say) and has horizontal length at least $12\pi/\delta$. By the slope estimate, the range of its imaginary part is at least $3\pi$ long, so it intersects some $\gamma_{j_0}^k$, so it intersects some $Z_j^k$, with $j_1 \leq j < j_2$, and so it intersects $B_j^k$. 

        Given an oblique line, under exponential it gets mapped to a spiral $\Sigma$, say. Every revolution, the spiral has two stretches where the slope lies in $(\delta/4, \delta/2)$, one in the right half-plane, one in the left half-plane. Let $(\Sigma_n)_{n \in \Z}$ denote the sequence of those stretches lying in the right half-plane, ordered so that the distance of $\Sigma_n$ from $0$ increases with $n$. For $n$ large enough, $\re(\Sigma_n) \subset (\omega_{j_0}, \infty)$ and the horizontal length of $\Sigma_n$ is arbitrarily large, in particular it can be taken bigger than $2 + 12\pi/\delta$. 
        Therefore it intersects some $B^k_j$. 

        Since $\exp$ of the line intersects $B^k_j$, $\exp \circ\exp\circ\exp$ of the line intersects $U$. This holds for every open set $U$ so the theorem is proven.
    \end{proof}
\section{An explicit dense curve}
    Given $R>1$, let $A_R$ denote the annulus $B(0,R)\setminus B(0, 1/R)$, the image of the vertical strip $H_R := \{z : \re(z) \in [-\log R, \log R)\}$ under $\exp$. 
        \begin{lem}\label{lem:2pidense}
            Let $\Sigma$ be a logarithmic spiral whose intersections with the imaginary axis occur at points $(w_k)_{k\in \Z}$, ordered by distance from $0$. 
            Then $\exp(\Sigma)$ is dense in $\ccc$ if and only if $(w_k/2\pi i)_{k\geq0}$ are dense modulo $1$. 
        \end{lem}
        \begin{proof}
            Denote by $\Sigma_k$ the connected component of the $\Sigma \cap H_R$ containing $w_k$.
            There exists $k_0$ for which, for all $k \leq k_0$, $\Sigma_k = \Sigma_{k_0}$, which spirals all the way in to $0$. The set $\exp(\Sigma_{k_0})$ has finite length and is not dense anywhere. Of course this then holds for $\exp(\Sigma_k)$ for each $k$, so we only need to consider positive $k$. 

            If $\Sigma = \exp(L_\alpha(p))$ say, denote by $Z_k$ the intersection of $H_R$ and the line which passes through $w_k$ with slope $-\alpha$. Then the Hausdorff distance of $\Sigma_k$ to $Z_k$ decreases to $0$ as $k \to +\infty$.
            
            Taken sequentially, the following statements are (clearly) equivalent. 
            \begin{itemize}
                \item
                    $(w_k/2\pi i)_{k\geq0}$ is dense modulo $1$. 
        \item
            the union of all $2\pi i$-translates of $\{Z_k\}_{k \geq 0}$ is dense in $H_R$.
            \item
                the union of all $2\pi i$-translates of $\{\Sigma_k\}_{k\geq 0}$ is dense in $H_R$. 
            \item
                $\bigcup_{k \geq 0} \exp(\Sigma_k)$ is dense in $A_R$. 
            \item
                $\exp(\Sigma)$ is dense in $\ccc$. 
        \end{itemize}
        This completes the proof of the lemma.
        \end{proof}
        
    \begin{proof}[Proof of Theorem~\ref{thm:explicit}]
        Let $a \in (0,1)$ have a binary expansion containing all possible strings of zeroes and ones; let $p :=\log (2\pi a) + \frac{\pi}2 i$ and $\alpha := \log 2 /2\pi$ as per Theorem~\ref{thm:explicit}. 
        By choice of $\alpha$, each time the imaginary part of the line $L_\alpha(p)$ increases by $2\pi$, the real part increases by $\log 2$. 
        Thus the intersections of the spiral $\Sigma :=\exp(L_\alpha(p))$ with the positive imaginary axis $(i\R^+$) occur at values $2\pi a 2^ki$, $k \in \Z$. 
        
        By choice of $a$, for all $k_0$ the set $\{2^ka\}_{k \geq k_0}$ is dense modulo 1. Now apply Lemma~\ref{lem:2pidense}.
    \end{proof}

    \section{Hausdorff dimension of the complementary set of parameters}
    In this section we prove Theorem~\ref{thm:two} and Corollary~\ref{cor:theta}.
    The Mass Distribution Principle is a standard source of lower bounds for the Hausdorff dimension. It is infused into the following lemma. 
    \begin{lem} \label{lem:mdp}
        Let $J$ be a non-degenerate interval, let $Y \subset J$ and let $\mu$ be a measure  with $\mu(Y) >0$. 
        For each $n \geq 1$, let $\cP_n$ be a finite partition of $J$ into intervals, each of length at most $2^{-n}$.
        Let $\eps \in (0, 1 )$, let $\beta >1$ and suppose 
        \begin{equation}\label{eq:1}
            \mu(P) \leq \beta (1+\eps)^n |P|
        \end{equation}
        for every $P \in \cP_n$.
        Then the Hausdorff dimension of $Y$ is at least $1 - 2\eps$.
        \end{lem}
        \begin{proof}
            For $r\in (0,1)$, let $n := \lceil -\log_2 r\rceil$. Let $x \in J$. If $P \in \cP_n$ then $|P| \leq 2^{-n} \leq r$, so if $P \cap B(x,r) \ne \emptyset$ then $P \subset B(x, 2r)$. 
            The total length of elements of $\cP_n$ intersecting $B(x,r)$ is thus at most $4r$. Summing \eqref{eq:1} over such elements, we deduce that
        $$\mu(B(x,r))/\beta \leq 4r (1+\eps)^n \leq 4r(1+\eps) e^{-\log(1+\eps)\log r/ \log 2} \leq 8r^{1 - \log(1+\eps)/\log 2} .$$ 
        Now $\log 2 > \frac12$ and $\log (1+\eps) < \eps$, so $$\mu(B(x,r))/8\beta \leq  r^{1-2\eps}.$$ 
        If $U_1, U_2, \ldots$ is any countable cover of $Y$ by balls of radius at most $1$, then
        $$\sum_{j\geq 1} |U_j|^{1-2\eps} \geq \sum_{j \geq 1} \mu(U_j)/8\beta \geq \mu(Y)/8\beta >0.$$ Since this positive lower bound does not depend on the cover, the Hausdorff dimension of $Y$ is at least $1 - 2\eps$, as required. 
    \end{proof}
    Together with the following lemma, one can glean an insight into the means of proving Theorem~\ref{thm:two}. 
    \begin{lem} \label{lem:angles}
        Let $I$ be an open subinterval of the imaginary axis  and let $\hat I := \bigcup_{k\in \Z}(2k\pi i + I)$ be the union of all $2\pi i$-translates of $\hat I$. Suppose $\hat I$ is disjoint from $B(0,1)$. Let $p \in \ccc$.
        Let $Y$ be a compact subset of $\arr$ 
        and suppose that $\exp (L_\alpha(p)) \cap \hat I = \emptyset$ for every $\alpha \in Y$. 
        Then there is an open set $U$ with $\exp \circ \exp (L_\alpha(p)) \cap U = \emptyset$ for each $\alpha \in Y$. 
    \end{lem}
    \begin{proof}
        Differentiating $t \mapsto \exp(p + t(i + \alpha))$ gives $(i+\alpha)\exp(p + t(i+\alpha))$. Thus $\exp(L_\alpha(p))$ has slope $-\alpha$ at each intersection with the imaginary axis. 
        Moreover, since $Y$ is bounded, there is a constant $C >1$ such that the slope of $\exp(L_\alpha(p))$ is bounded in absolute value by $C$ in the region $$\left\lbrace z : |\re(z)| < \frac12, |\im(z)| > \frac12 \right\rbrace.$$ 
	Let $D$ denote the body of the rhombus with diagonal $I$ and sides of slope~$\pm C$, and $\widehat{D}$ the union of all $2\pi i$-translates of $D$. 
	Then $\exp(L_\alpha(p)) \cap \widehat{D} = \emptyset$ for each $\alpha \in Y$. 
        Let $x$ be the midpoint of $I$ and denote by $U$ the open set $\exp(B(x, |I|/4C)$. By construction, $B(x, |I|/4C) \subset D$ so $\exp^{-1}(U) \subset \widehat{D}$.
	Thus $\exp\circ\exp(L_\alpha(p)) \cap U = \emptyset$ for each $\alpha \in Y$, as required. 
    \end{proof}

    Now we can prove Theorem~\ref{thm:two}, which states that 
        for each $p \in \ccc$ and each open set $X \subset \arr$, the set $\{\alpha \in X  : \overline{ \exp\circ \exp(L_\alpha(p))} \ne \ccc \}$ has Hausdorff dimension $1$. 
        \begin{proof}[Proof of Theorem~\ref{thm:two}]
        We can assume $0 \notin X$. 
        Writing $\sigma$ for the map sending points to their complex conjugates, $\exp \circ\, \sigma = \sigma \circ \exp$ and $\sigma(L_\alpha(p)) = L_{-\alpha}(\sigma(p))$ so,  without loss of generality (replacing $p$ by $\sigma(p)$ and $X$ by $-X$, if necessary), one can assume $X \subset \arr^+$. 

        Given $X$ and $p$, let $X' = (\alpha_0, \alpha_1)$ be a non-degenerate subinterval of $X$ with $0 < \alpha_0 < \alpha_1$. Let $\xi : \alpha \mapsto p + i + \alpha$ and let $J$ be the line segment $\xi(X')$. 
        For $k \in \Z$, let
        $S_k := (k + \frac12 ) \pi i + \arr$.
        Then $\exp(S_k)$ is a vertical ray leaving $0$, heading up if $k$ is even and down if $k$ is odd. 
        Let $\phi_k$ be the central projection with respect to $p$ from $J$ to $S_k$, so 
        $$\phi_k(p + i + \alpha) = \re(p) + \left((k + \frac12)\pi - \im(p)\right) \alpha +  i \left((k + \frac12)\pi - \im(p)\right).$$
        In particular, as a map from $J$ to $S_k$, $\phi_k$ is affine with derivative 
        $$D\phi_k(z) = (k + \frac12 ) \pi - \im(p)$$ for every $z \in J$. 
        There exists a  (possibly negative) $k_0 \in \Z$ such that, for all $k \leq k_0$, $\phi_k(J) \subset \{z : \re(z) <0\}$, 
        and thus, for $k \leq k_0$, $\exp \circ\, \phi_k(J) \subset B(0,1)$. 
        Writing $\psi_k := \exp \circ \,\phi_k$ on $J$, $\psi_k$ maps $J$ onto a subinterval of the imaginary axis, see Figure~\ref{fig:two}. As $\phi_k$ is affine,  the distortion of $\psi_k$ on an interval $W \subset J$ is bounded by $\exp(|\phi_k(W)|)$. 
\begin{figure}
    \centering
    \def\svgwidth{0.8\columnwidth}
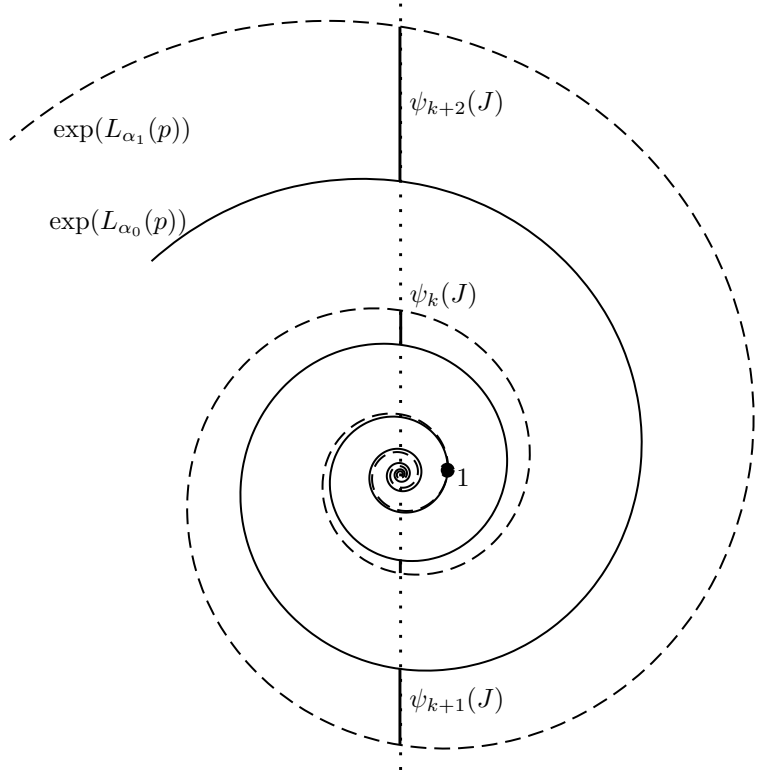
    \caption{Two logarithmic spirals ($\exp(L_{\alpha_0}(p))$ and $\exp(L_{\alpha_1}(p))$, drawn with $p=0$) and the increasing (in length) subintervals $\psi_k(J), \psi_{k+1}(J), \psi_{k+2}(J)$ of the imaginary axis.}
    \label{fig:two}
\end{figure}
	We have $|D\psi_k| = |D\exp (\phi_k)| |D\phi_k| = |D\phi_k| \exp(\re(\phi_k))$, so 
        \begin{equation} \label{eq:der}
            |D\psi_k(p + i + \alpha)| = \left|\left(k + \frac12 \right) \pi - \im(p)\right| e^{\re(p)} e^{(\frac{\pi}{2}- \im(p))\alpha} e^{k\pi \alpha}.
        \end{equation}
        Thus for $k > |\im(p)|/\pi$, 
        \begin{equation}\label{eq:api}
        |D\psi_{k+1}(p+ i+ \alpha)/D\psi_k(p+ i+ \alpha)| > e^{\alpha\pi},  
        \end{equation}
        so the derivatives grow exponentially. 
        Moreover, there exists $C \in (0,1)$ such that, for each $k \geq k_0$ with $p \notin S_k$, 
        \begin{equation}\label{eq:c1}
            |D\psi_k| > C.
        \end{equation}

        \begin{remark} \textbf{Choice of the constant $N$:} If $\alpha>0$ is small, then there is not much expansion at each revolution. We shall consider blocks of $N$ (half-) revolutions at a time, for large integers $N$. Let $I$ be small open sub-interval of the imaginary axis and let $\hat I := \bigcup_{k\in \Z} (2k \pi i  + I)$.  Let $\eps >0$ and  suppose that $V$ is a subinterval of $J$, that  $\psi_{j}(V) \cap \hat I = \emptyset$ for $j \leq nN$ and that $\psi_{nN}(V) \geq \eps$. 
            We shall obtain estimates for the points in $V$ not meeting $\hat I$ for $j \leq (n+1)N$. To continue by induction, we will need to regain the starting condition length $\geq \eps$. 
    By~\eqref{eq:api}, 
        $|\psi_{nN+j}(V)| \geq e^{j\alpha_0\pi}\eps$, and for $j \leq N$ the length is bounded by $L := |\psi_{(n+1)N}(V)|$. 
    Note $L \geq e^{N\alpha_0\pi}\eps$.
    The number of connected components of the set of points $z \in V$ with $\psi_{nN+j}(z)\notin \hat I$ for $j = 1, \ldots, N$ is bounded by $N (L+2)$ [if $L >2\pi$, one can improve the bound to $NL/2\pi +1$]. 
    The proportion of points $z \in V$ with $\psi_{nN+j}(z)\notin \hat I$ for $j = 1, \ldots, N$ is at least $1 - 2N |I|/\eps$, if one assumes bounded distortion giving a factor of $2$. 
    If we remove all connected components whose image under $\phi_{(n+1)N}$ is less than $\eps$, the remaining proportion is at least $1-2N|I|/\eps - 2\eps N(L+2) / L$. 
    If one
    takes $\eps = N^{-2}$, $|I| = N^{-4}$ 
    and $N$  large, 
    then $L>1$ and 
    the proportion is at least $1 -8/N$, 
    which can be made as close to $1$ as we desire. To get good starting conditions for a forthcoming induction argument, $N$ may need to be taken larger again, and $|I|$ slightly smaller. 
    \end{remark}
    \medskip

     Let an integer $N > 2|k_0| +8\pi$ be large enough that 
        \begin{itemize}
            \item 
                $N\pi > 2|\im(p)|$; 
            \item
                $e^{N\pi \alpha_0 /2} > N^4$;
            \item
                $N e^{\re(p)} > 1;$
            \item
                $1/N^2 < |J|$.
        \end{itemize}
        By \eqref{eq:der} and choice of $N$, 
            for all $z = p+ i+\alpha \in J$,
        \begin{equation}\label{eq:der2}
            |D\psi_{N}(z)| > (N\pi /2) e^{\re(p)} e^{N\pi \alpha_0 /2} > N^4.
        \end{equation}
            From \eqref{eq:api} and choice of $N$, we obtain
        \begin{equation}\label{eq:der3}
            |D\psi_{(n+1)N}(z) / D\psi_{nN}(z)| >   N^4
        \end{equation}
        for each $n \geq 1$ and $z \in J$. 

        Let $M := \sup_{z\in J} |D\psi_N(z)|$, so for any subinterval $J' \subset J$, $|\psi_N(J')| \leq M|J'|$. 
        Let $I$ be an open subinterval of the imaginary axis of length $N^{-4}C/M$ whose $2\pi i$-translates are disjoint from $\exp(p)$ and from $B(0,1)$. Let $\hat I := \bigcup_{k\in \Z} (2k \pi i  + I)$. 
        For $k \leq k_0$, $\psi_k(J) \subset B(0,1)$, so $\psi_k(J) \cap \hat I = \emptyset$. 

        Let $J'$ be a subinterval of $J$ of length $1/N^2$. 
        Let $J_n$ be the set of points $z \in J'$ for which $\psi_k(z) \notin \hat I$ for all $k \leq n$. 
	Note that $J_{k_0} = J'$. 

        We now  deal with the steps from $k_0$ to  $N$, to get a good starting interval. We shall later use induction to pass from $nN$ to $(n+1)N$. 
        For 
         $k = k_0+1, \ldots, N$, 
         $$|\phi_k(J')| < |J'| ((N+\frac12)\pi - \im(p)) < N^{-2}((N+1)\pi/2) < \pi/N.$$   
         Hence the distortion of $\psi_k$ is bounded by $e^{\pi/N} < 2$. 
         For $k = k_0 +1, \ldots, N$, $|\psi_k(J')| \leq |\psi_N(J')|$ and by \eqref{eq:der2}, $|\psi_N(J')| > N^4/N^2 =N^2$. 
         For $k\leq N$, 
        the number of connected components of $\hat I$ intersecting $\psi_k(J')$ is bounded by $|\psi_N(J')|$; 
         it follows that $m(\hat I \cap \psi_k(J')) \leq |\psi_N(J')| |I|.$  
        Using \eqref{eq:c1} and then choice of $M$ and $I$,
        \begin{equation*}
        \begin{split}
            m(J_N) &= |J'| - m\left(J' \cap \bigcup_{k = k_0+1}^N \psi_k^{-1}(\hat I )\right) \\ &\geq |J'| -  (N-k_0) |\psi_N(J')||I| /C 
            \\ & \geq |J'| - (N-k_0) |J'| N^{-4}
        \\& > |J'|/2,
        \end{split}
    \end{equation*}
    say.
        Meanwhile, $J_N$ has at most $(N - k_0)|\psi_N(J')|$ connected components. 
        Therefore, at least one connected component $V$ of $J_N$ must satisfy 
        $$|V| > |J'|/3(N-k_0) |\psi_N(J')|$$
        and, more importantly (by the distortion bound of $2$), 
        $$|\psi_N(V)| > 1/2(N-k_0) > 1/ N^2.$$ 
        Let $\cW_1 := \{V\}$. 

        Now we repeat the argument for general intervals. 
        Let us define $\cW_{n}$ inductively as follows. For $W \in \cW_n$, let $\cA_W$ denote the (finite) collection of connected components $A$ of $J_{(n+1)N} \cap W$ for which $|\psi_{(n+1)N}(A)| \geq 1/N^2.$ 
        Let 
        $$\cW_{n+1} := \cup_{W\in \cW_n} \cA_W.$$
        Note $\cW_1 = \{V\}$ is non-empty. 
        The set 
        $$\Lambda := \bigcap_{n\geq 1} \bigcup_{W\in \cW_n} W$$
        is a closed subset of $J$, as a countable intersection of finite unions of closed sets. For $z \in \Lambda$, $z \in J_k$ for all $k$, so the image of the line passing through $p$ and $z$ is a spiral which avoids $\hat I$. We shall show that $\Lambda$ is non-empty and has dimension at least $1 - 10/N$. 

        For $W \in \cW_n$, let $W^+ := \cup_{A \in \cA_W}A.$
        In order to apply Lemma~\ref{lem:mdp}, we will need to show that $m(W^+)/m(W)$ is close to $1$; in particular it will be at least $1-4/N$. 

        Since $W \in \cW_n$, $|\psi_{nN}(W)| \geq 1/N^2$. 
        Let $k$ satisfy $nN \leq k (n+1)N$. 
        By \eqref{eq:api},   $\psi_k(W)$ has length at least $1/N^2$.
        Hence 
        \begin{equation} \label{eq:mipsi}
            m(\hat I \cap \psi_k(W))/|\psi_k(W)| \leq N^2 |I|.
        \end{equation}
        Now $$|D\phi_k|/|D\phi_{nN}| = ((k+\frac12)\pi - \im(p))/((nN+\frac12)\pi - \im(p)) < 4.$$
        Since $\psi_{nN}(W) \cap \hat I = \emptyset$, one obtains $|\psi_{nN}(W)| \leq 2\pi$ and  $\phi_{nN}(W)$ has length (crudely) bounded by $1/8$. Hence
        $|\phi_k(W)|$ is bounded by $1/2$. Therefore the distortion of $\psi_k$ on $W$ is bounded by $e^{1/2} < 2$. 
        We deduce from this and \eqref{eq:mipsi} ($N$ times, for $k=nN+1, \ldots, (n+1)N$) that 
        the set $Z := J_{(n+1)N} \cap W$  satisfies 
        $m(Z)/|W| \geq 1 - 2N^3|I|$. 
        Meanwhile, by \eqref{eq:der3}, 
        $$|\psi_{(n+1)N}(W)| \geq N^4 |\psi_{nN}(W)| \geq N^4/ N^{2} = N^2.$$ [As an aside, note that the image is  long and  therefore contains many components of $\hat I$, so elements of $\cA_W$ will have length much less than $|W|/2$.]
        The set $Z$  (see Figure~\ref{fig:wj}) has at most $N |\psi_{(n+1)N}(W)|$ connected components. Those of length at least  $2|W|/|\psi_{(n+1)N}(W)| N^2$ get mapped by 
        $\psi_{(n+1)N}$ onto an interval of length at least $1/N^2$, by bounded distortion, so they are contained in $\cA_W$. Knowing a bound for the number of connected components, we deduce that those of length at most 
         $2|W|/ |\psi_{(n+1)N}(W)|N^2$ 
        have measure bounded by $2|W|/N$. 
        Consequently, 
        \begin{align*}
            m(W^+)/m(W) &\geq 1 - 2N^3|I| - 2/N
          \\
        &> 1 - 4/N,
    \end{align*}
    noting $|I| \leq 1/N^4.$ 
    \begin{figure}
    \centering
    \def\svgwidth{0.8\columnwidth}
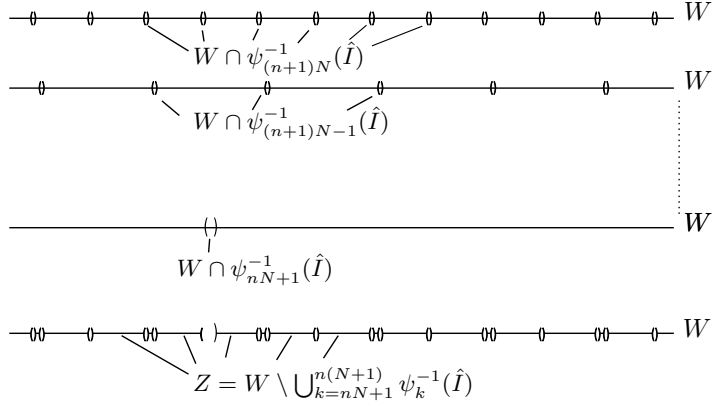
    \caption{A schematic drawing of $Z = W \setminus \bigcup_{k=nN+1}^{(n+1)N}\psi_k^{-1}(\hat I)$ showing multiple copies of $W$. Connected components of $W \cap \psi_k^{-1}(\hat I)$ are tiny, so most of $Z$ will consist of relatively large connected components.}
    \label{fig:wj}
\end{figure}

    Since $|J'| = 1/N^2$,   $|V|< 1/2$ for (the unique interval) $V \in \cW_1$. It follows that for each $W \in \cW_n$,   $|W| \leq 2^{-n}$. 

    Recall we wish to construct a measure on $\Lambda= \cap_{n\geq 1} \cup_{W\in \cW_n} W$, in order to estimate its dimension using Lemma~\ref{lem:mdp}.
    For each $n \geq 1$, let us introduce a  measure $\mu_n$ on $\bigcup_{W \in \cW_n} W$. Let $\mu_1$ be Lebesgue measure restricted to the unique interval $V \in \cW_1$. 
    Define inductively $\mu_n$, for $n \geq 2$, as follows. 
    For each $W \in \cW_{n-1}$, set  
    \begin{equation}\label{eq:mun}
        \mu_n := \frac{m(W)}{m(W^+)}\mu_{n-1}
    \end{equation}
        on $W^+$, and $\mu_n:= 0$ on $W \setminus W^+$.
    As defined,  $\mu_n(W^+) = \mu_{n-1}(W)$ for each $W \in \cW_{n-1}$, whence  $\mu_k(W) = \mu_n(W)$ for all $k \geq n$ and each $W \in \cW_n$. 
    Since also $\max_{W \in \cW_n}|W| \leq 2^{-n}$, there exists a unique (weak) limit measure $$\mu :=\lim_{n\to \infty} \mu_n$$ and $\mu$ is supported on $\Lambda$ with $\mu(\Lambda) = \mu_n(J') = |V|$. We need to check the limit measure is well-behaved. In particular, it should not have atoms. 
    By induction using~\eqref{eq:mun}, $$\mu_n(W) \leq |W| (1-4/N)^{-n+1}$$ for $W \in \cW_n$. 
    Thus for $z \in \Lambda$ and $n \geq 1$, 
    there are at most two elements $W_1,W_2 \in \cW_n$ intersecting all tiny neighbourhoods of $z$, and 
    $$\mu_n(W_i) \leq |W_i| (1-4/N)^{-n+1} \leq 2^{-n/2 +1}$$
    for $i=1,2$. Hence $\mu_k(W_i) \leq 2^{-n/2 +1}$ for all $k \geq n$, and so 
     $\mu(\{z\}) \leq 2^{-n/2 + 2}$ for each $n$; therefore $\mu$ is continuous (i.e.\ it has no atoms). 
     Since $\mu$ is continuous, $\mu(W) = \mu_k(W)$ for each $W \in \cW_n$ and $k \geq n$. 

    We are nearly at a stage where we can apply Lemma~\ref{lem:mdp}. For each $n$, let $\cQ_n$ denote a finite partition
    of $J' \setminus \bigcup_{W \in \cW_n} W$ into intervals such that each $Q \in \cQ_n$ has $|Q| < 2^{-n}$. For each $Q \in \cQ_n$, $\mu_k(Q) = 0$ for all $k \geq n$, hence $\mu(Q) = 0$ (using continuity of $\mu$). 
    Let $$\cP_n := \cQ_n \cup \cW_n,$$ so $\cP_n$ is a partition of $J'$. 
    From the construction, $$\mu(P) \leq |P| (1-4/N)^{-n} \leq |P|(1+5/N)^n$$ for each $n \geq 1$ and $P \in \cP_n$. 
    By Lemma~\ref{lem:mdp}, the Hausdorff dimension of $\Lambda$ is at least $1 - 10/N$. 
    Recalling $\Lambda \subset J'$, set  $Y := \xi^{-1}(\Lambda) \subset X'$. 
    Applying Lemma~\ref{lem:angles} to $Y$, we obtain that for each $\alpha \in Y$, $\exp \circ \exp(L_\alpha(p))$ is not dense. As $\xi$ is a translation it preserves Hausdorff dimension, and the dimension of $Y$ is at least $1 - 10/N$. But $N$ could be taken arbitrarily large (of course, $I$ and therefore $Y$ depend on choice of $N$). Noting that any set with subsets of dimension arbitrarily close to $1$ has dimension at least $1$, the proof of Theorem~\ref{thm:two} is complete. 
    \end{proof}

    \begin{proof}[Proof of Corollary~\ref{cor:theta}]
        Taking $p = 2\pi i$, the intersections of the spiral with the positive imaginary axis occur at points $\exp(2\pi\alpha k)2\pi i$, $k \in \Z$. From Theorem~\ref{thm:two} and Lemma~\ref{lem:2pidense}, we deduce that the set of $\alpha$ in any open interval $X$ for which $\exp(2\pi \alpha k)$ is not dense modulo $1$ has dimension $1$, from which the result follows (taking $X = (\log I)/2\pi$). 
\end{proof}
\begin{remark} One could use Lemma~\ref{lem:2pidense} to prove Theorem~\ref{thm:one} (using Koksma's theorem \cite{Koksma:dense}), however the lemma cannot be used to prove Theorem~\ref{thm:threesteps}, neither does Theorem~\ref{thm:threesteps} provide information about distributions of sequences modulo $1$. \end{remark}
    \section{Distribution}
    Given $\alpha, p$ and  the corresponding spiral $\Sigma : t \mapsto \exp (p + t(i + \alpha))$, we set  $\rho := \exp \circ \Sigma$, a  parametrisation of $\exp\circ \exp$ of the line $L_\alpha(p)$. We now study the distribution of $\rho(t)$. 
For every measurable set $A$ and $T>1$, let 
$$\mu_T(A):= \frac1{2T} m(\{t \in [-T,T]: \rho(t) \in A\}),$$ 
where $m$ denotes Lebesgue measure. Then $\mu_T$ is a probability measure. 
    \begin{proof}[Proof of Theorem~\ref{thm:weak}]
        We can assume without loss of generality that $\alpha > 0$. 
        Since  $ 
        \lim_{t \to -\infty} |\Sigma(t)| = 0,$ 
        $$ \lim_{t \to -\infty} \rho(t) = 1.$$ 
        Let us define intervals $$I^+_n := 2n\pi + [-\pi/2 + 1/n - \im(p), \pi/2 + 1/n - \im(p)].$$ The intervals are chosen so that for $t \in I^+_n$ and $n$ large, 
        $$\re(\Sigma(t)) \geq \sin(1/n) \exp(\re(p) + 2n\pi \alpha - \pi/2 - \im(p)) \gg 1,$$ so 
        $$
        \lim_{n \to \infty} |\rho(I_n^+)| = +\infty.$$
        Setting $I_n^- := I_n^+ + \pi$, we similarly obtain that 
        $$
        \lim_{n \to \infty} |\rho(I_n^-)| = 0.$$

        Noting that the intervals $I^\pm_n$ have length approaching $\pi$, and the spaces between them have length $\approx 2/n$, it follows that 
        $$\lim_{T \to \infty} \mu_T = \delta_1/2 + \frac{\delta_0 + \delta_\infty}4,$$
        as required. 
    \end{proof}
%


\bibliography{expexpbib} 
\bibliographystyle{plain}
  
\end{document}